\documentclass[a4paper, 11pt,reqno]{amsart}

\usepackage[T1]{fontenc}      
\usepackage[foot]{amsaddr}    

\usepackage[british]{babel}   
\usepackage{csquotes}         
\usepackage[babel]{microtype} 

\usepackage{mathtools}        
\usepackage{amssymb}          
\usepackage{amsthm}           
\usepackage{amstext}          
\usepackage{mleftright}       
\usepackage{gensymb}          
\usepackage[makeroom]{cancel} 
\usepackage{mathdots}         
\usepackage{mathrsfs}         
\usepackage{dsfont}           

\usepackage{stickstootext}
\usepackage[stickstoo,varbb]{newtxmath} 
\makeatletter
\DeclareFontFamily{U}{ntxmia}{\skewchar \font =127}
 \DeclareFontShape{U}{ntxmia}{m}{it}{
                        <-> \ntxmath@scaled ntxmia
                      }{}    
                      \DeclareFontShape{U}{ntxmia}{b}{it}{
                        <-> \ntxmath@scaled ntxbmia
                      }{}
\makeatother

\usepackage{graphicx}         
\usepackage{psfrag}           
\usepackage{caption}          
\usepackage{tikz}             
\usetikzlibrary{backgrounds}

\usepackage{booktabs}         

\usepackage{enumitem}         

\usepackage[margin=1in]{geometry} 

\usepackage[textsize=footnotesize]{todonotes}

\usepackage[numbers,sort]{natbib}

\makeatletter
\def\NAT@spacechar{~}
\makeatother

\usepackage{hyperref}              
\usepackage[capitalise]{cleveref}  
\hypersetup{colorlinks=true,
   citecolor=blue,
   filecolor=blue,
   linkcolor=blue,
   urlcolor=blue
}
\usepackage{url}

\makeatletter
\if@cref@capitalise
\crefname{figure}{figure}{figures}
\crefname{claim}{Claim}{Claims}
\else
\crefname{figure}{Figure}{Figures}
\crefname{claim}{claim}{claims}
\fi
\makeatother
\Crefname{figure}{Figure}{Figures}
\Crefname{claim}{Claim}{Claims}

\usepackage{algorithm}
\usepackage{algpseudocode}

\allowdisplaybreaks

\theoremstyle{definition}
\newtheorem{definition}{Definition}

\theoremstyle{plain}
\newtheorem{claim}{Claim}

\newtheorem{proposition}[definition]{Proposition}
\newtheorem{theorem}[definition]{Theorem}

\newtheorem{lemma}[definition]{Lemma}

\newtheorem{conjecture}[definition]{Conjecture}
\newtheorem{question}[definition]{Question}

\newenvironment{claimproof}{%
\let\origqed=\qedsymbol%
\renewcommand{\qedsymbol}{$\blacktriangleleft$}%
\begin{proof}}{\end{proof}\let\qedsymbol=\origqed}

\renewcommand{\binom}[2]{\ensuremath{\mleft(\kern-.1em\genfrac{}{}{0pt}{}{#1}{#2}\kern-.1em\mright)}}    
\newcommand{\inbinom}[2]{\ensuremath{\bigl(\kern-.1em\genfrac{}{}{0pt}{}{#1}{#2}\kern-.1em\bigr)}} 

\newcommand*\nume{\ensuremath{\mathrm{e}}}

\makeatletter
\def\moverlay{\mathpalette\mov@rlay}
\def\mov@rlay#1#2{\leavevmode\vtop{%
  \baselineskip\z@skip \lineskiplimit-\maxdimen
  \ialign{\hfil$\m@th#1##$\hfil\cr#2\crcr}}}
\newcommand{\charfusion}[3][\mathord]{
    #1{\ifx#1\mathop\vphantom{#2}\fi
        \mathpalette\mov@rlay{#2\cr#3}
      }
    \ifx#1\mathop\expandafter\displaylimits\fi}
\makeatother

\newcommand{\eps}{\epsilon}

\newcommand{\COMMENT}[1]{}
\newcommand{\COMNEW}[1]{}

\title{Seymour's second neighbourhood conjecture:\linebreak{} random graphs and reductions}

\author[A.~Espuny D\'iaz]{Alberto Espuny D\'iaz}
\email{espuny-diaz@informatik.uni-heidelberg.de}
\address[Espuny D\'iaz, Granet]{Institut f\"ur Informatik, Universit\"at Heidelberg, 69120 Heidelberg, Germany.}
\author[A.~Gir\~ao]{Ant\'onio Gir\~ao}
\email{girao@maths.ox.ac.uk}
\address[Gir\~ao, Kronenberg]{Mathematical Institute, University of Oxford, Andrew Wiles Building, Radcliffe Observatory Quarter, Woodstock Road, Oxford, United Kingdom.}
\author[B.~Granet]{Bertille Granet}
\email{granet@informatik.uni-heidelberg.de}
\author[G.~Kronenberg]{Gal Kronenberg}
\email{kronenberg@maths.ox.ac.uk}

\thanks{Alberto Espuny Díaz was partially funded by the Deutsche Forschungsgemeinschaft (DFG) through projects no.\ 447645533 and 513704762.
Ant\'onio Gir\~ao is supported by ERC Advanced Grant 883810.
Bertille Granet was partially supported by the DFG project 428212407 and by the Alexander von Humboldt Foundation.
Gal Kronenberg is supported by the Royal Commission for the Exhibition of 1851.}

\date{\today}

\begin{document}

\begin{abstract}
A longstanding conjecture of Seymour states that in every oriented graph there is a vertex whose second outneighbourhood is at least as large as its outneighbourhood. 
In this short note we show that, for any fixed $p\in[0,1/2)$, a.a.s.\ every orientation of $G(n,p)$ satisfies Seymour's conjecture (as well as a related conjecture of Sullivan).
This improves on a recent result of Botler, Moura and Naia.
Moreover, we show that $p=1/2$ is a natural barrier for this problem, in the following sense: for any fixed $p\in(1/2,1)$, Seymour's conjecture is actually equivalent to saying that, with probability bounded away from $0$, every orientation of $G(n,p)$ satisfies Seymour's conjecture.
This provides a first reduction of the problem. 

For a second reduction, we consider minimum degrees and show that, if Seymour's conjecture is false, then there must exist arbitrarily large strongly-connected counterexamples with bounded minimum outdegree.
Contrasting this, we show that vertex-minimal counterexamples must have large minimum outdegree.
\end{abstract}

\maketitle

\section{Introduction}\label{sec:intro}

Given any digraph $D$ and a vertex $v\in V(D)$, let us write $N_1(v)\coloneqq\{w\in V\setminus \{v\}:vw\in E(D)\}$ for the \emph{outneighbourhood} of $v$, and let $N_2(v)\coloneqq\{w\in V\setminus (\{v\}\cup N_1(v)) : E(N_1(v),w)\neq\varnothing\}$ denote the \emph{second outneighbourhood} of $v$.
We say that $v$ is a \emph{Seymour vertex} if $|N_2(v)|\geq|N_1(v)|$.
In this note, the word \emph{graph} refers to finite, simple graphs.
Seymour conjectured the following (see~\cite{DL95}).

\begin{conjecture}\label{conj:Seymour}
Every oriented graph has a Seymour vertex.
\end{conjecture}

Seymour's conjecture, which is closely tied to a particular case of the well-known conjecture of \citet{CH78}, has attracted a lot of attention due to its simplicity, and yet it remains widely open.
The statement has only been proved for some special classes of oriented graphs.
Of particular interest is the case of tournaments, where the conjecture is known to hold: solving a conjecture of Dean~\cite{DL95}, this was first proved by \citet{F96}, and later reproved by \citet{HT00}.
Other special classes of graphs or variants of the conjecture have been considered in different papers (see, e.g.,~\cite{AGGWYZ23,CSY03,FY07,Sea15,CC23,Gha12,Lla13} and the references therein).

\subsection{Orientations of random graphs}

In this note we focus primarily on the case of orientations of random graphs.
The first results in this direction are due to \citet{CGHZ16}, who considered random orientations of the binomial random graph $G(n,p)$ (that is, a graph on $n$ vertices where each possible edge is included independently with probability $p$).
They showed that, when $C_1(\log n/n)^{1/2}\leq p\leq1-C_2(\log n/n)^{1/2}$, a.a.s.\ (that is, with probability tending to $1$ as $n\to\infty$) a random orientation of $G(n,p)$ contains a Seymour vertex.
This implies, in particular, that almost every oriented graph has a Seymour vertex.
A more interesting approach was taken by \citet{BMN22}, who considered \emph{all} orientations of the random graph $G(n,p)$.
They showed that a.a.s.\ \emph{every} orientation of $G(n,p)$ contains a Seymour vertex whenever $p<1/4$.
Here we improve the range of $p$ for which this holds.

\begin{theorem}\label{thm:below}
Let\/ $p<1/2$.
Then, a.a.s.\ every orientation of\/ $G(n,p)$ contains a Seymour vertex.
\end{theorem}

Our methods also apply to a related conjecture of Sullivan \cite{Sullivan}; see \cref{app:Sullivan}.

Moreover, it turns out that $p=1/2$ is a natural barrier for this problem.
Indeed, we show that, if we could extend \cref{thm:below} to any value of $p>1/2$, then we would prove \cref{conj:Seymour} in full generality.
This means that proving \cref{conj:Seymour} in general is as hard as proving that it holds for the random graph $G(n,p)$ for $p\in(1/2,1)$, providing a first reduction of the problem.

\begin{theorem}\label{thm:equiv}
    Let\/ $p\in(1/2,1)$.
    If \cref{conj:Seymour} is false, then a.a.s.\ there is an orientation of\/ $G(n,p)$ with no Seymour vertex.
\end{theorem}

\subsection{Counterexamples and minimum degrees}

We show that, if \cref{conj:Seymour} is false, then a vertex-minimal counterexample must have high minimum outdegree. 

\begin{proposition}\label{prop:degree-minimal}
    Let $D$ be a vertex-minimal counterexample to \cref{conj:Seymour}.
    Then, $\delta^{+}(D)>\sqrt{|V(D)|}$.
\end{proposition}

On the other hand, the following result essentially says that, if there exist any counterexamples to \cref{conj:Seymour}, then there must exist arbitrarily large strongly-connected counterexamples with bounded minimum outdegree. This gives a second reduction of Seymour's conjecture.

\begin{proposition}\label{prop:degree-reduction}
    Suppose \cref{conj:Seymour} is false.
    Then, for every function $d=d(n)=\omega(1)$, there exist infinitely many $n\in\mathbb{N}$ for which there exists an $n$-vertex strongly-connected oriented graph $D$ with $\delta^+(D)<d$ which contains no Seymour vertex.
\end{proposition}

\section{Proofs}

Our notation will be standard.
For any positive integer $n$, we write $[n]\coloneqq\{1,\ldots,n\}$.
Throughout, for the sake of readability, we assume that $n$ is sufficiently large when needed and ignore rounding for asymptotic statements.
Given a graph $G=(V,E)$, for any disjoint sets $A,B\subseteq V$ we write $e(A,B)$ for the number of edges with one endpoint in $A$ and the other in $B$.
Given an orientation $\vec{G}$ of $G$, we will write $\vec{e}(A,B)$ to denote the number of (oriented) edges from $A$ to $B$.
Moreover, for any vertex $v\in V$, we write $d^+(v)\coloneqq|N_1(v)|$ for the outdegree of $v$ in $\vec{G}$.
Given a set $A\subseteq V(D)$, we define $N_1(A)\coloneqq(\bigcup_{v\in A}N_1(v))\setminus A$.
The graphs or oriented graphs to which the notation refers will always be clear from the context.

\subsection{Orientations of random graphs}

We will use the following version of Chernoff's bound (see, e.g.,~\cite[Corollary~2.3]{JLR}).

\begin{lemma}\label{lem:Chernoff}
Let\/ $X\sim\mathrm{Bin}(n,p)$ be a binomial random variable.
Then, for all\/ $0<\delta<1$ we have that\/
$\mathbb{P}[|X-np|\geq\delta np]\leq2\nume^{-\delta^2np/3}$.
\end{lemma}

We begin with the proof of \cref{thm:below}.

\begin{proof}[Proof of \cref{thm:below}]
Consider $G(n,p)$.
By the result of \citet[Theorem~2]{BMN22}, we may assume, e.g., that $p\geq1/8$ (our proof works for any constant $p\in(0,1/2)$).
We will need some probabilistic estimates. 
Let $\eps\coloneqq1/2-p$, so $\eps\in(0,3/8]$.
Let $\delta\coloneqq\eps/3$.
Let $C=C(\eps)$ be sufficiently large, and let $C'\coloneqq2C/\eps$. 
We will use the following properties.

\begin{claim}\label{claim1}
A.a.s.\ $G\sim G(n,p)$ satisfies the following properties:
\begin{enumerate}[label=$(\mathrm{\alph*})$]
    \item\label{claim1item1} Every vertex\/ $v\in V(G)$ satisfies that\/ $d(v)\leq (1-\eps)n/2$. 
    \item\label{claim1item3} For every pair of disjoint sets\/ $A,B\subseteq V(G)$ with\/ $|A|,|B|\geq C\log n$ we have that
    \[(p-\delta)|A||B|\leq e(A,B)\leq(1-\eps)|A||B|/2.\]
\end{enumerate}
\end{claim}

\begin{claimproof}
    Property \ref{claim1item1} is standard, and follows by a direct application of \cref{lem:Chernoff} and a union bound over all vertices.\COMMENT{For any fixed vertex $v$ we have that
    \begin{align*}
        \mathbb{P}[d(v)> (1-\eps)n/2]=\mathbb{P}[d(v)> (p+\eps/2)n]\leq\mathbb{P}[d(v)\neq (p\pm\eps/2)n]\leq\mathbb{P}\left[d(v)\neq\left(1\pm\frac{101\eps}{200p}\right)\mathbb{E}[d(v)]\right]\leq\nume^{-\Theta(n)}.
    \end{align*}
    The conclusion follows by a union bound over all vertices.}
    Property \ref{claim1item3} follows similarly, though the union bound is a bit more involved, so we include the details.
    
    It suffices to show that a.a.s.\ $e(A,B)=(p\pm\delta)|A||B|$ for all the required pairs of sets.
    Fix any disjoint sets $A,B\subseteq V(G)$ with $|A|,|B|\geq C\log n$.
    By \cref{lem:Chernoff}, we have that
    \[\mathbb{P}[e(A,B)\neq(p\pm\delta)|A||B|]=\mathbb{P}[e(A,B)\neq(1\pm\delta/p)p|A||B|]\leq2\exp(-\delta^2|A||B|/3p).\]
    We may assume without loss of generality that $|A|\leq|B|$.
    By a union bound over all choices of~$A$ and~$B$, the probability $P$ that the statement fails satisfies
    \begin{align*}
        P&\leq2\sum_{a=C\log n}^{n/2}\sum_{b=a}^n\binom{n}{a}\binom{n}{b}\exp\left(-\frac{\delta^2}{3p}ab\right)\\
        &\leq\sum_{a=C\log n}^{n/2}\sum_{b=a}^nn^an^b\exp\left(-\frac{\delta^2}{3p}ab\right)=\sum_{a=C\log n}^{n/2}\sum_{b=a}^n\exp\left((a+b)\log n-\frac{\delta^2}{3p}ab\right).
    \end{align*}
    By making the constant $C$ sufficiently large,\COMMENT{It suffices to take $C\geq12p/\delta^2$} we may ensure that $a\log n,b\log n\leq\delta^2ab/12p$, and so
    \[P\leq\sum_{a=C\log n}^{n/2}\sum_{b=a}^n\exp\left(-\frac{\delta^2}{6p}ab\right)\leq\Theta(n^2)n^{-\Theta(\log n)}=o(1).\qedhere\]
\end{claimproof}

Now, let $G$ be any $n$-vertex graph satisfying the properties of \cref{claim1}, and let $\vec{G}$ be an arbitrary orientation of $G$. 
We are going to show that, if $n$ is sufficiently large, then $\vec{G}$ contains a Seymour vertex.
Let $x_1,\ldots,x_n$ be a labelling of $V(G)$ such that for all $i,j\in[n]$ with $i\leq j$ we have that $d^{+}(x_i)\geq d^{+}(x_j)$; that is, the labels are assigned by decreasing order of the outdegrees of the vertices.
Let $X_0\coloneqq\varnothing$ and, for each $i\in[n]$, let $X_i\coloneqq\{x_j:j\in[i]\}$.

\begin{claim}\label{claim2}
If there is some\/ $i\in[n]$ such that\/ $|N_1(x_i)\cap X_{i-1}|\geq C'\log n$, then\/ $x_i$ is a Seymour vertex.
\end{claim}

\begin{claimproof}
Assume there is some $i\in[n]$ satisfying the condition.
Let $A\coloneqq N_1(x_i)\cap X_{i-1}$, so $|A|\geq C'\log n$.
Let $m\coloneqq|N_1(x_i)|$ (and note that $m\geq|A|$).
Let $S\subseteq A$ be the set of all vertices $v\in A$ with $\vec{e}(v,N_1(x_i))\geq(1+\eps)m/2$, and let $T\subseteq A$ be the set of all vertices $v\in A$ with $\vec{e}(v,V\setminus (\{x_i\}\cup N_1(x_i)))\geq(1-\eps)m/2$.
Note that, by definition,
\begin{equation}\label{eq:bound1}
    \vec{e}(T,N_2(x_i))=\vec{e}(T, V\setminus (\{x_i\}\cup N_1(x_i)))\geq|T|(1-\eps)m/2.
\end{equation}
Observe that $|S|<C\log n$ (indeed, if we assume otherwise, then for any set $S'\subseteq S$ of size $|S'|=C\log n$ we have that\COMMENT{For the third inequality, we use the fact that $C\log n=\eps C'\log n/2\leq\eps m/2$ by the definition of $C'$.}
\[e(S',N_1(x_i)\setminus S')\geq\sum_{v\in S'}\vec{e}(v,N_1(x_i)\setminus S')\geq |S'|((1+\eps)m/2-C\log n)\geq |S'|m/2,\]
which contradicts \cref{claim1}~\ref{claim1item3}).\COMMENT{Note that, by assumption, $|S'|\geq C\log n$, and we also have $|N_1(x_i)\setminus S'|\geq (C'-C)\log n\geq C\log n$, so \cref{claim1}~\ref{claim1item3} applies.}
Since $A\setminus S\subseteq T$, this implies that $|T|\geq C\log n$.\COMMENT{Indeed, we have that $A\setminus S\subseteq T$, since every vertex $v\in A$ satisfies $d^+(v)\geq m$ by the definition of $A$, and any such $v$ with $\vec{e}(v,N_1(x_i))<(1+\eps)m/2$ must therefore satisfy $\vec{e}(v,V\setminus (\{x_i\}\cup N_1(x_i)))\geq m-((1+\eps)m/2-1)-1=(1-\eps)m/2$.
Therefore, $|T|\geq C'\log n-C\log n\geq C\log n$.}
It thus follows from \cref{claim1}~\ref{claim1item3} that\COMMENT{Formally, to apply \cref{claim1}~\ref{claim1item3} we must have $|N_2(x_i)|\geq C\log n$. 
We have not proved this explicitly, but this is a consequence of \eqref{eq:bound1}. 
Indeed, we just showed that $|T||N_2(x_i)|\geq\vec{e}(T,N_2(x_i))\geq|T|(1-\eps)m/2$, so $|N_2(x_i)|\geq(1-\eps)m/2\geq C\log n$.
(The last bound holds since $m\geq C'\log n$ and by the definition of $C'$.
Note that $T$ and $N_2(x_i)$ are disjoint by definition.)}
\begin{equation}\label{eq:bound2}
    \vec{e}(T,N_2(x_i))\leq(1-\eps)|T||N_2(x_i)|/2.
\end{equation}
Combining \eqref{eq:bound1} and \eqref{eq:bound2}, it follows that $x_i$ is a Seymour vertex.
\end{claimproof}

By \cref{claim2}, we may assume that $\vec{G}$ satisfies that 
\begin{equation}\label{eq:keyprop}
\text{for all $i\in[n]$ we have $|N_1(x_i)\cap X_{i-1}|\leq C'\log n$.}
\end{equation} 
Now let $B\coloneqq X_{2\delta n}$ be partitioned into $B_1\coloneqq X_{\delta n}$ and $B_2\coloneqq B\setminus B_1$.
By \cref{claim1}~\ref{claim1item3}, we have that $e(B_1,B_2)\geq(p-\delta)|B_1||B_2|=(p-\delta)\delta^2 n^2$.
By \eqref{eq:keyprop}, for each $v\in B_2$ we have $\vec{e}(v,B_1)\leq C'\log n$, and so $\vec{e}(B_1,B_2)\geq(p-2\delta)\delta^2n^2$.
It follows by averaging that there is some vertex $y\in B_1$ such that $\vec{e}(y,B_2)\geq(p-2\delta)\delta n$.
Let $Y\coloneqq N_1(y)\cap B_2$, and let $Z$ be the set of vertices $z\in V\setminus B$ such that $\vec{e}(Y,z)=0$.

We claim that $|Z|<C\log n$.
Indeed, assume that $|Z|\geq C\log n$, let $Z'\subseteq Z$ with $|Z'|=C\log n$, and partition $Y$ arbitrarily into $\Theta(n/\log n)$ sets of size at least $C\log n$.
By \cref{claim1}~\ref{claim1item3}, there are edges of $G$ between $Z'$ and each of these sets, so together $e(Y,Z')=\Omega(n/\log n)$.
By averaging, there must be some vertex $z\in Z'$ with $e(Y,z)=\Omega(n/\log^2n)$.
Since $\vec{e}(Y,z)=0$ by definition, we must have that $\vec{e}(z,Y)=\Omega(n/\log^2n)$.
However, this contradicts \eqref{eq:keyprop}.

Now, since $V\setminus (B\cup Z)\subseteq N_1(y)\cup N_2(y)$, it follows that $|N_1(y)\cup N_2(y)|\geq (1-3\delta)n$.
By \cref{claim1}~\ref{claim1item1}, we have that $|N_1(y)|\leq(1-\eps)n/2$, and hence $|N_2(y)|\geq (1-3\delta)n-(1-\eps)n/2\geq|N_1(y)|$, so $y$ is a Seymour vertex.
\end{proof}

Next, we prove \cref{thm:equiv}.

\begin{proof}[Proof of \cref{thm:equiv}]
    Let $D$ be an oriented graph for which \cref{conj:Seymour} fails, and let $H$ be its underlying graph.
    Let $h\coloneqq|V(H)|$.
    We say that a graph $G$ on $n\geq h$ vertices has a \emph{good ordering} if there is a labelling $x_1,\ldots, x_n$ of its vertices satisfying the following properties:
    \begin{enumerate}[label=(P\arabic*)]
        \item\label{goodordering1} $G[\{x_1,\ldots,x_h\}]$ induces a copy of $H$, and
        \item\label{goodordering2} for every $i\in\{h+1,\ldots,n\}$ we have $|N(x_i)\cap\{x_1,\ldots,x_{i-1}\}|\geq i/2$.
    \end{enumerate}
    
    We claim that, if $G$ has a good ordering $x_1,\ldots, x_n$, then there is an orientation of $G$ with no Seymour vertex.
    Indeed, consider the orientation $\vec{G}$ of $G$ where $\vec{G}[\{x_1,\ldots,x_h\}]$ induces a copy of $D$ (which is possible by \ref{goodordering1}) and all other edges are oriented towards the vertex with smaller label in the ordering.
    It is clear that none of $x_1,\ldots,x_h$ can be Seymour vertices in $\vec{G}$: by construction, their outneighbourhoods are defined only by (a subgraph isomorphic to) $D$, so if they were a Seymour vertex for $\vec{G}$, they would also be a Seymour vertex for $D$.
    Moreover, for any $i\in\{h+1,\ldots,n\}$, both $N_1(x_i)$ and $N_2(x_i)$ are contained in $\{x_1,\ldots,x_{i-1}\}$.
    By \ref{goodordering2}, we are guaranteed that $|N_1(x_i)|\geq i/2$, and so $|N_2(x_i)|<|N_1(x_i)|$.

    It remains to show that a.a.s.\ $G(n,p)$ has a good ordering.
    We will make use of the following standard properties.

    \begin{claim}\label{claim:equiv}
        Let\/ $\eps\coloneqq p-1/2$.
        A.a.s.\ $G\sim G(n,p)$ satisfies the following properties.
        \begin{enumerate}[label=$(\mathrm{\roman*})$]
            \item\label{basicproperty0} $G$ contains an induced copy of\/ $H$.
            \item\label{basicproperty1} Every vertex\/ $v\in V(G)$ satisfies\/ $d(v)\geq(1+\eps/2)n/2$.
            \item\label{basicproperty2} For every set\/ $X\subseteq V(G)$ with\/ $|X|\leq(1-\eps/4)n$, there exists a vertex\/ $v\in V(G)\setminus X$ such that\/ $e_G(v,X)\geq|X|/2$.
        \end{enumerate}
    \end{claim}

    \begin{claimproof}
        Property \ref{basicproperty0} is well known (see, e.g., \cite[Proposition~11.3.1]{Diestel}).
        Property \ref{basicproperty1} is standard and follows by an application of \cref{lem:Chernoff} to the degree of a vertex and a union bound over all vertices.\COMMENT{Fix a vertex $v$.
        By \cref{lem:Chernoff}, the probability that $v$ has too small degree is bounded by
        \[\mathbb{P}[d(v)<(1+\eps/2)n/2]=\mathbb{P}[d(v)<(1-3\eps/4p)np]\leq\mathbb{P}[d(v)<(1-4\eps/5p)(n-1)p]\leq2\nume^{-\frac{16\eps^2}{75p}(n-1)}=\nume^{-\Theta(n)}.\]
        A union bound over all vertices completes the proof.}
        Here we only show the details for the (also simple) proof of property~\ref{basicproperty2}.
        Let 
        \[a\coloneqq \frac{3\log16}{p\left(1-\frac{1}{2p}\right)^2}\]
        (note that this is independent of $n$).
        
        Consider first all sets $X\subseteq V(G)$ of size $|X|<a$.
        We first show that a.a.s.\ all vertices in each such set have a common neighbour in $G$.
        Indeed, fix a set $X$ with $|X|=k<a$.
        The probability that any given vertex $v\in V(G)\setminus X$ is a neighbour of all vertices in $X$ is $p^k$, which is a constant.
        Thus, the probability that no vertex is a common neighbour of $X$ is $(1-p^k)^{n-|X|}\leq\nume^{-p^kn/2}=o(n^{-2k})$.
        By a union bound over all sets $X\subseteq V(G)$ with $|X|=k$ and over the possible values of $k$, we reach the desired conclusion.

        Consider now any set $X\subseteq V(G)$ with $a\leq|X|\leq n/2$.
        For any vertex $v\in V(G)\setminus X$, we have by \cref{lem:Chernoff} that $\mathbb{P}[e(v,X)<|X|/2]\leq1/8$.\COMMENT{The variable $e(v,X)$ is binomial, so Chernoff bounds apply. 
        We have 
        \[\mathbb{P}[e(v,X)<|X|/2]=\mathbb{P}\left[e(v,X)<\left(1-\left(1-\frac{1}{2p}\right)\right)|X|p\right]\leq2\exp\left(-\frac13\left(1-\frac{1}{2p}\right)^2|X|p\right)\leq1/8.\]
        Here, the first inequality comes from \cref{lem:Chernoff}, and the second is a consequence of the lower bound on $|X|$.}
        Thus, the probability that all vertices $v\in V(G)\setminus X$ have fewer than $|X|/2$ neighbours in $X$ is at most $8^{-(n-|X|)}\leq2^{-3n/2}$.\COMMENT{We are using the upper bound on $|X|$.}
        The conclusion follows by a union bound over the at most $2^n$ choices for $X$.

        Lastly, fix any set $X\subseteq V(G)$ with $n/2\leq|X|\leq(1-\eps/4)n$.
        Similarly as above, for any vertex $v\in V(G)\setminus X$ we have by \cref{lem:Chernoff} that $\mathbb{P}[e(v,X)<|X|/2]\leq\nume^{-\Theta(n)}$,\COMMENT{Proceeding as above, we have
        \[\mathbb{P}[e(v,X)<|X|/2]\leq2\exp\left(-\frac13\left(1-\frac{1}{2p}\right)^2|X|p\right),\]
        and now this is exponential in $n$ by the bounds on $|X|$.} so the probability that all vertices $v\in V(G)\setminus X$ have fewer than $|X|/2$ neighbours in $X$ is at most $\nume^{-\Theta(n^2)}$.\COMMENT{Here we use the upper bound on $|X|$.}
        One last union bound completes the proof of the claim.
    \end{claimproof}

    Condition now on the event that $G\sim G(n,p)$ satisfies the properties of \cref{claim:equiv}, which holds a.a.s.
    We construct a good ordering for $G$ as follows.
    First, choose an arbitrary set of $h$ vertices which induces a copy of $H$ (which exists by \ref{basicproperty0}) and label the vertices $x_1,\ldots,x_h$ arbitrarily.
    Thus, \ref{goodordering1} is satisfied.
    Next, for each $i\in\{h+1,\ldots,(1-\eps/4)n\}$ in turn, by \ref{basicproperty2}, there exists some vertex $v\in V(G)\setminus\{x_1,\ldots,x_{i-1}\}$ satisfying that $|N(v)\cap\{x_1,\ldots,x_{i-1}\}|\geq i/2$; let $x_i$ be an arbitrary vertex satisfying this property.
    Finally, by \ref{basicproperty1}, all vertices $v\in V(G)\setminus\{x_1,\ldots,x_{(1-\eps/4)n}\}$ satisfy that $d(v)\geq(1+\eps/2)n/2$, and thus, $|N(v)\cap\{x_1,\ldots,x_{(1-\eps/4)n}\}|\geq n/2$.
    Thus, the remaining vertices can be labelled arbitrarily while guaranteeing that \ref{goodordering2} holds.
\end{proof}

\subsection{Counterexamples and minimum degrees}

We begin with the simple proof of \cref{prop:degree-reduction}.

\begin{proof}[Proof of \cref{prop:degree-reduction}]
    Let $D_0$ be a counterexample to \cref{conj:Seymour}.
    We define $n_0\coloneqq|V(D_0)|$ and $d_0\coloneqq\delta^+(D_0)$.
    Let $N_0\in\mathbb{N}$ be such that $d(n)>d_0+n_0$ for all $n\geq N_0n_0$.
    For any $N\geq N_0$, we construct a strongly-connected counterexample $D$ to \cref{conj:Seymour} with $Nn_0$ vertices and $\delta^+(D)<d(Nn_0)$.

    The construction is a ``blow-up'' of a consistently-oriented cycle where each vertex is replaced by a copy of $D_0$.
    Formally, let $D_0,D_1,\ldots,D_{N-1}$ be vertex-disjoint copies of $D_0$ and, for each $i\in[N]$, add a complete bipartite graph between the vertices of $D_i$ and those of $D_{i+1}$, with all edges oriented towards $D_{i+1}$ (where we take indices modulo $N$).
    Note that the resulting oriented graph $D$ is strongly connected (simply by considering the complete bipartite oriented graphs between the different copies of $D_0$) and that $\delta^+(D)=d_0+n_0<d(Nn_0)$, as desired.
    By the symmetry of the construction, we now only need to verify that no vertex in $D_1$ is a Seymour vertex for $D$.
    Let us write $N'(\cdot)$ to denote neighbourhoods in $D_1$, while we use $N(\cdot)$ for neighbourhoods in $D$.
    Note that, for every $v\in V(D_1)$, we have that $N_1(v)=N_1'(v)\cup V(D_2)$ and $N_2(v)=N_2'(v)\cup V(D_3)$, and so $|N_2(v)|\geq|N_1(v)|$ if and only if $|N_2'(v)|\geq|N_1'(v)|$.
    But $D_1$ contains no Seymour vertices, so neither does $D$.
\end{proof}

Now note that any vertex-minimal counterexample $D$ to \cref{conj:Seymour} must be strongly connected.
Indeed, otherwise one may consider the auxiliary directed acyclic graph obtained from the strongly-connected components of $D$ and note that any strong component which has no outneighbours in this auxiliary graph would itself be a smaller counterexample to the conjecture.

\begin{proof}[Proof of \cref{prop:degree-minimal}]
Suppose for a contradiction that the statement is false, that is, that there is a vertex-minimal counterexample $D$ with $n$ vertices and $d\coloneqq\delta^+(D)\leq\sqrt{n}$.
Pick a vertex $x\in V(D)$ of smallest outdegree, i.e., $d^+(x)=d$, and let $A_1\coloneqq\{x\}\cup N_1(x)$, $X_1\coloneqq N_2(x)$, and $B_1\coloneqq V(D)\setminus(A_1\cup X_1)$.
Note that, by assumption, $x$ is not a Seymour vertex, so $|X_1|<d$.
Moreover, $E(A_1,B_1)=\varnothing$.

Now we go through the following iterative process:
while $X_i$ contains some non-empty subset $X_i'$ with $|N_1(X_i')\cap B_i|<|X_i'|$, define $A_{i+1}\coloneqq A_i\cup X_i'$, $X_{i+1}\coloneqq (X_i\setminus X_i')\cup(N_1(X_i')\cap B_i)$ and $B_{i+1}\coloneqq B_i\setminus N_1(X_i')$, and proceed to the next iteration.
Observe that, throughout the process, we maintain the property that $E(A_i,B_i)=\varnothing$.
In particular, since $D$ is strongly connected, this guarantees that, if $B_i\neq\varnothing$, then $E(X_i,B_i)\neq\varnothing$.
In the $i$-th iteration of the above procedure we decrease the size of $X_i$ by at least $1$, and decrease the size of $B_i$ by at most $|X_i|$. 
Hence, after at most $d-1$ steps, the procedure must stop, say, at step $t<d$ with $X_t\neq\varnothing$ and $|B_t|\geq n-\inbinom{d}{2}\geq1$ by the choice of $d$.\COMMENT{So this comes out to $2n\geq d^2-d+2$, which is equivalent to $d\leq(1+(8n-7)^{1/2})/2$, and one can verify that $\sqrt{n}\leq(1+(8n-7)^{1/2})/2$ for all $n\geq1$.}

By the minimality of $D$, $D[A_t]$ must contain a Seymour vertex, say $z$.
We claim that $z$ is also a Seymour vertex for $D$.
Indeed, let us write $N'(\cdot)$ to denote neighbourhoods in $D[A_t]$.
Then, using that $E(A_t,B_t)=\varnothing$, we have that 
\[|N_2(z)|-|N_1(N_1(z)\cap X_t)\cap B_t|\geq|N_2'(z)|\geq|N_1'(z)|=|N_1(z)|-|N_1(z)\cap X_t|.\]
Since the procedure above has ended, we have that $|N_1(N_1(z)\cap X_t)\cap B_t|\geq|N_1(z)\cap X_t|$, and so it follows that $|N_2(z)|\geq|N_1(z)|$.
This contradicts the fact that $D$ has no Seymour vertex.
\end{proof}

\section{Concluding remarks}

In \cref{thm:below} we showed that, for fixed $p<1/2$, a.a.s.\ $G(n,p)$ satisfies Seymour's conjecture, and in \cref{thm:equiv} we showed that, for fixed $p\in(1/2,1)$, \cref{conj:Seymour} is equivalent to knowing that, with probability bounded away from $0$, all orientations of the random graph $G(n,p)$ contain a Seymour vertex. 
It would be interesting to see whether the latter provides a useful avenue for tackling \cref{conj:Seymour}.
We note that \citet{BMN22} showed that, for fixed $p\in[1/2,2/3)$, a.a.s.\ every orientation of $G(n,p)$ with minimum outdegree $\omega(\sqrt{n})$ has a Seymour vertex, so for $p$ in this range it now suffices to study orientations with small minimum outdegree.

Neither of our results applies when $p=(1\pm o(1))/2$, which is arguably one of the more interesting values for $p$.
In particular, it would be interesting to understand whether every orientation of almost all graphs (which corresponds precisely to the case $p=1/2$) contains a Seymour vertex.
In the spirit of applying a similar method as in our proof of \cref{thm:equiv}, we ask the following question.

\begin{question}
Does\/ $G(n,1/2)$ a.a.s.\ have an ordering\/ $x_1,\ldots, x_n$ of its vertices such that\/ $|N(x_i)\cap$\linebreak{} $\{x_1,\ldots x_{i-1}\}|\geq i/2$ for all\/ $i\in[n]$? 
\end{question}

\Cref{thm:equiv} also leaves a gap for the cases when $p=1-o(1)$.
It would thus be interesting to consider orientations of $G(n,p)$ when $p$ tends to $1$ (with the case $p=1$, corresponding to tournaments, being solved already~\cite{F96,HT00}).

Lastly, \cref{prop:degree-reduction,prop:degree-minimal} provide information about the minimum degree of (assumed) counterexamples to \cref{conj:Seymour}.
\Cref{prop:degree-reduction} shows that, if \cref{conj:Seymour} is false, then we must be able to construct arbitrarily large strongly-connected counterexamples with bounded minimum outdegree.
\Cref{prop:degree-minimal}, on the other hand, shows that the search for (vertex-minimal) counterexamples may be limited to oriented graphs with large minimum outdegree.

\section*{Acknowledgements}

We are grateful to the anonymous referees for helpful comments on an earlier version of this paper.

\bibliographystyle{mystyle}
\bibliography{references}

\begin{thebibliography}{18}
\newcommand{\enquotenew}[1]{`#1'}
\providecommand{\natexlab}[1]{#1}
\providecommand{\url}[1]{\texttt{#1}}
\providecommand{\urlprefix}{URL }
\providecommand{\doi}[1]{\textsc{doi}:
  \href{https://doi.org/#1}{\nolinkurl{#1}}}
\providecommand*{\eprint}[2][]{arXiv:
  \href{https://arxiv.org/abs/#2}{\nolinkurl{#2}}}

\bibitem[{{Ai}, {Gerke}, {Gutin}, {Wang}, {Yeo} and {Zhou}(2023)}]{AGGWYZ23}
J.~{Ai}, S.~{Gerke}, G.~{Gutin}, S.~{Wang}, A.~{Yeo} and Y.~{Zhou}, {{On
  Seymour's and Sullivan's second neighbourhood conjectures}}. \emph{arXiv
  e-prints}  (2023). \eprint{2306.03493}.

\bibitem[{Alon and Spencer(2016)}]{AS16}
N.~Alon and J.~H. Spencer, \emph{The probabilistic method}. Wiley Series in
  Discrete Mathematics and Optimization, John Wiley \& Sons, Inc., Hoboken, NJ,
  4th ed. (2016), ISBN 978-1-119-06195-3.

\bibitem[{{Botler}, {Moura} and {Naia}(2023)}]{BMN22}
F.~{Botler}, P.~F.~S. {Moura} and T.~{Naia}, {{Seymour's second neighborhood
  conjecture for orientations of (pseudo)random graphs}}. \emph{Discrete
  Mathematics} 346.12 (2023),  113\,583, \doi{10.1016/j.disc.2023.113583}.

\bibitem[{Caccetta and H\"{a}ggkvist(1978)}]{CH78}
L.~Caccetta and R.~H\"{a}ggkvist, {On minimal digraphs with given girth}.
  \emph{Proceedings of the {N}inth {S}outheastern {C}onference on
  {C}ombinatorics, {G}raph {T}heory, and {C}omputing ({F}lorida {A}tlantic
  {U}niv., {B}oca {R}aton, {F}la., 1978)}, \emph{Congress. Numer.}, vol. XXI,
  Utilitas Math., Winnipeg, MB (1978)  181--187.

\bibitem[{Chen and Chang(2023)}]{CC23}
B.~Chen and A.~Chang, {A note on {Seymour}'s second neighborhood conjecture}.
  \emph{Discrete Appl. Math.} 337 (2023),  272--277,
  \doi{10.1016/j.dam.2023.05.012}.

\bibitem[{Chen, Shen and Yuster(2003)}]{CSY03}
G.~Chen, J.~Shen and R.~Yuster, {Second neighborhood via first neighborhood in
  digraphs}. \emph{Ann. Comb.} 7.1 (2003),  15--20,
  \doi{10.1007/s000260300001}.

\bibitem[{Cohn, Godbole, {Wright Harkness} and Zhang(2016)}]{CGHZ16}
Z.~Cohn, A.~Godbole, E.~{Wright Harkness} and Y.~Zhang, {The number of
  {Seymour} vertices in random tournaments and digraphs}. \emph{Graphs Comb.}
  32.5 (2016),  1805--1816, \doi{10.1007/s00373-015-1672-9}.

\bibitem[{Dean and Latka(1995)}]{DL95}
N.~Dean and B.~J. Latka, {Squaring the tournament---an open problem}.
  \emph{Congr. Numerantium} 109 (1995),  73--80.

\bibitem[{{Diestel}(2017)}]{Diestel}
R.~{Diestel}, \emph{{Graph theory}}, \emph{{Graduate Texts in Mathematics}},
  vol. 173. Springer, Berlin (2017), ISBN 978-3-662-53621-6; 978-3-662-53622-3.

\bibitem[{Erd\H{o}s and R\'{e}nyi(1959)}]{ER59}
P.~Erd\H{o}s and A.~R\'{e}nyi, {On random graphs. {I}}. \emph{Publ. Math.
  Debrecen} 6 (1959),  290--297.

\bibitem[{Fidler and Yuster(2007)}]{FY07}
D.~Fidler and R.~Yuster, {Remarks on the second neighborhood problem}. \emph{J.
  Graph Theory} 55.3 (2007),  208--220, \doi{10.1002/jgt.20229}.

\bibitem[{Fisher(1996)}]{F96}
D.~C. Fisher, {Squaring a tournament: {A} proof of {Dean}'s conjecture}.
  \emph{J. Graph Theory} 23.1 (1996),  43--48,
  \doi{10.1002/(SICI)1097-0118(199609)23:1<43::AID-JGT4>3.0.CO;2-K}.

\bibitem[{Ghazal(2012)}]{Gha12}
S.~Ghazal, {Seymour's second neighborhood conjecture for tournaments missing a
  generalized star}. \emph{J. Graph Theory} 71.1 (2012),  89--94,
  \doi{10.1002/jgt.20634}.

\bibitem[{Havet and Thomass{\'e}(2000)}]{HT00}
F.~Havet and S.~Thomass{\'e}, {Median orders of tournaments: {A} tool for the
  second neighborhood problem and {Sumner}'s conjecture}. \emph{J. Graph
  Theory} 35.4 (2000),  244--256,
  \doi{10.1002/1097-0118(200012)35:4<244::AID-JGT2>3.0.CO;2-H}.

\bibitem[{Janson, \L{}uczak and Ruci\'{n}ski(2000)}]{JLR}
S.~Janson, T.~\L{}uczak and A.~Ruci\'{n}ski, \emph{Random graphs}.
  Wiley-Interscience Series in Discrete Mathematics and Optimization,
  Wiley-Interscience, New York (2000), \doi{10.1002/9781118032718}.

\bibitem[{Llad{\'o}(2013)}]{Lla13}
A.~Llad{\'o}, {On the second neighborhood conjecture of {Seymour} for regular
  digraphs with almost optimal connectivity}. \emph{Eur. J. Comb.} 34.8 (2013),
   1406--1410, \doi{10.1016/j.ejc.2013.05.023}.

\bibitem[{Seacrest(2015)}]{Sea15}
T.~Seacrest, {The arc-weighted version of the second neighborhood conjecture}.
  \emph{J. Graph Theory} 78.3 (2015),  219--228, \doi{10.1002/jgt.21800}.

\bibitem[{{Sullivan}(2006)}]{Sullivan}
B.~D. {Sullivan}, {{A summary of problems and results related to the
  Caccetta-H\"{a}ggkvist conjecture}}. \emph{arXiv Mathematics e-prints}
  (2006). \eprint{math/0605646}.

\end{thebibliography}


\appendix

\section{On Sullivan's second neighbourhood conjecture for random graphs}\label{app:Sullivan}

\citet{Sullivan} proposed multiple variants and strengthenings of \cref{conj:Seymour}.
Here we consider one of them.
Given any digraph $D$ and a vertex $v\in V(D)$, let us now write $N^+(v)\coloneqq N_1(v)$ for the outneighbourhood of $v$, $N^-(v)\coloneqq\{w\in V\setminus \{v\}:wv\in E(D)\}$ for the inneighbourhood of~$v$, and $N^+_2(v)\coloneqq N_2(v)$ for the second outneighbourhood of $v$.
Given a set $A\subseteq V(D)$, we define $N^+(A)\coloneqq(\bigcup_{v\in A}N^+(v))\setminus A$.
We say that $v$ is a \emph{Sullivan vertex} if $|N^+_2(v)|\geq|N^-(v)|$.
(Notice that, contrary to the definition of a Seymour vertex, $N^+_2(v)$ and $N^-(v)$ may intersect.)

\begin{conjecture}[{\cite[Conjecture~6.6]{Sullivan}}]\label{conj:Sullivan}
    Every oriented graph has a Sullivan vertex.
\end{conjecture}

\Cref{conj:Sullivan} has received far less attention than \cref{conj:Seymour}.
Only recently, \citet{AGGWYZ23} showed that it holds for certain classes of graphs, including tournaments, planar oriented graphs and some families of oriented split graphs.
Moreover, they also showed that almost all oriented graphs satisfy \cref{conj:Sullivan}, as a.a.s.\ a random orientation of $G(n,p)$ (for $p\in(0,1)$ independent of $n$) satisfies the conjecture.
Here we use similar ideas to those presented in the paper to prove a result analogous to \cref{thm:below} for Sullivan's conjecture.

\begin{theorem}\label{thm:Sullivan_below}
    Let\/ $p=p(n)\in[0,1]$ be such that\/ $\limsup_{n\to\infty}p<1/2$.
    Then, a.a.s.\ every orientation of\/ $G(n,p)$ contains a Sullivan vertex.
\end{theorem}

As a difference with respect to \cref{thm:below}, we note that here we must consider the cases when $p$ tends to $0$; for Seymour's conjecture, these cases were covered by the earlier results of \citet{BMN22}.

We note first that, if a graph $G$ has an isolated vertex, then so do all orientations of $G$, and isolated vertices are Sullivan vertices.
Thus, it follows from the well-known threshold for connectivity~\cite{ER59} that, if $p\leq(1-\eps)\log n/n$,
then a.a.s.\ every orientation of $G(n,p)$ has a Sullivan vertex.
Thus, in order to prove \cref{thm:Sullivan_below} we only need to consider larger values of $p$.
In particular, \cref{thm:Sullivan_below} follows directly from the following two results.

\begin{proposition}\label{prop:Sullivan1}
    Let\/ $3\log n/4n\leq p=p(n)= o(1)$.
    Then, a.a.s.\ every orientation of\/ $G(n,p)$ contains a Sullivan vertex.
\end{proposition}

\begin{proposition}\label{prop:Sullivan2}
    Let\/ $p=p(n)\in[0,1]$ with\/ $p=\omega(\log^{-1/3}n)$ and\/ $\limsup_{n\to\infty}p<1/2$.
    Then, a.a.s.\ every orientation of\/ $G(n,p)$ contains a Sullivan vertex.
\end{proposition}

We will need to bound large deviations for the upper tail of some binomial random variables.
For this, we will use the following Chernoff bound (see, e.g., \cite[Corolary~A.1.10]{AS16}).\COMMENT{Let $Y_1, \dots Y_n$ be mutually independent with $\mathbb{P}[Y_i=1-p]=p$ and $\mathbb{P}[Y_i=-p]=1-p$ for each $i$. Let $Y=Y_1+\dots+Y_n$. For $i\in [n]$, let $X_i\coloneqq Y_i+p$ and $X\coloneqq X_1+\dots+X_n$. Then, \cite[Corollary~A.1.10]{AS16} applied with $a=\delta np$ gives
\[\mathbb{P}[X\geq(1+\delta)np]=\mathbb{P}[Y\geq \delta np]\leq \exp(\delta np-(1+\delta)np\log(1+\delta))=(\nume^\delta/(1+\delta)^{1+\delta})^{np}.\]}

\begin{lemma}\label{eq:Chernoff_strong}
    Let\/ $X\sim\mathrm{Bin}(n,p)$ be a binomial random variable.
    Then, for all\/ $\delta>0$ we have that\/
    $\mathbb{P}[X\geq(1+\delta)np]\leq (\nume^\delta/(1+\delta)^{1+\delta})^{np}$.
\end{lemma}

\begin{proof}[Proof of \cref{prop:Sullivan1}]
    We begin our proof by observing that a.a.s.\ $G(n,p)$ satisfies certain properties.

    \begin{claim}\label{claim:Sulli1}
        A.a.s.\ $G\sim G(n,p)$ satisfies the following properties:
        \begin{enumerate}[label=$(\mathrm{\alph*})$]
            \item\label{SullItem1} For each pair of distinct vertices $u,v\in V(G)$, there are fewer than $np/2400$ paths of length~$2$ joining $u$ and $v$.
            \item\label{SullItem2} Every vertex $v\in V(G)$ satisfies $d(v)\leq4np$.
            \item\label{SullItem3} For every set $A\subseteq V(G)$ of $n/3$ vertices we have $e(A)\geq n^2p/25$.
        \end{enumerate}
    \end{claim}

    \begin{claimproof}
        All claims follow from direct applications of \cref{lem:Chernoff,eq:Chernoff_strong} (with \cref{eq:Chernoff_strong} being needed for \ref{SullItem1} and \ref{SullItem2}) and a union bound.
    \end{claimproof}

    Now let $G$ be an arbitrary $n$-vertex graph satisfying the properties from \cref{claim:Sulli1}.
    Note that, if in an orientation $\vec{G}$ of $G$ there is some vertex $v\in V(G)$ such that $|N_2^+(v)|>4np$, then trivially $v$ must be a Sullivan vertex, since by \cref{claim:Sulli1}~\ref{SullItem2} all vertices have indegree at most $4np$.
    We are going to show that, in fact, every orientation $\vec{G}$ of $G$ contains at least one such vertex, thus deriving the desired result.

    We argue by contradiction.
    Suppose that there is an orientation $\vec{G}$ of $G$ such that every vertex $v\in V(G)$ satisfies $|N_2^+(v)|\leq4np$.
    We are going to count consistently oriented paths of length~$2$ in two different ways to reach a contradiction.
    Let $P$ denote the number of such paths.
    First, we may count~$P$ by adding over all vertices the number of paths starting at that vertex.
    For each vertex, the number of such paths is less than $n^2p^2/300$: indeed, any such path must have its endpoint in $N^+(v)\cup N^+_2(v)$, which is a set of size at most $8np$ by \cref{claim:Sulli1}~\ref{SullItem2} and the assumption on the orientation, and by \cref{claim:Sulli1}~\ref{SullItem1} there are fewer than $np/2400$ paths joining any given pair of vertices.
    Thus, 
    \begin{equation}\label{equa:Sulli1}
        P<n^3p^2/300.
    \end{equation}

    Next, we may obtain $P$ by adding over all vertices $v\in V(G)$ the number of paths of length $2$ whose middle vertex is $v$. 
    Observe that, for each fixed $v$, this number is $d^-(v)d^+(v)$.
    We claim that at least $2n/3$ vertices have outdegree at least $np/10$.
    Indeed, if we assume otherwise, there is a set $A\subseteq V(G)$ of $n/3$ vertices of outdegree at most $np/10$, so by \cref{claim:Sulli1}~\ref{SullItem3} we have that
    \[\frac{n^2p}{25}\leq e(A)\leq \sum_{v\in A}d^+(v)\leq\frac{n}{3}\frac{np}{10}=\frac{n^2p}{30},\]
    a contradiction.
    Similarly, at least $2n/3$ vertices must have indegree at least $np/10$.
    But this means that at least $n/3$ vertices have both in- and outdegree at least $np/10$, which immediately implies that 
    \[P\geq n^3p^2/300.\]
    This results in a contradiction with \eqref{equa:Sulli1}.
\end{proof}

\begin{proof}[Proof of \cref{prop:Sullivan2}]
Consider $G(n,p)$.
As $\limsup_{n\to\infty}p<1/2$, there exist $n_0\in\mathbb{N}$ and a constant $\eps>0$ such that $p\leq1/2-\eps$ for all $n\geq n_0$.
Let $C=C(\eps)$ be sufficiently large and let $\delta\coloneqq \eps/5$.
We will use the following properties, which follow from standard applications of \cref{lem:Chernoff}.\COMMENT{For property \ref{claim1bisitem1}, for any fixed vertex $v$ we have that
    \begin{align*}
        \mathbb{P}[d(v)\neq (1\pm \delta)np]&\leq \mathbb{P}[d(v)\neq (1\pm\delta/2)(n-1)p]\leq 2\exp(-\delta^2(n-1)p/12)
        =\exp(-\omega(n^{2/3}))=o(n^{-1}).
    \end{align*}
    The conclusion follows by a union bound over all vertices.}\COMMENT{For property \ref{claim1bisitem2}, fix a set $A$ with the correct size.
    Then,
    \[\mathbb{P}\left[e(A)\neq(1\pm\delta)\binom{|A|}{2}p\right]\leq2\exp(-\delta^2|A|^2p/7)=\exp(-\Omega(n^{5/3}))=o(2^{-n}).\]
    So the result follows by a union bound on the at most $2^n$ choices for $A$.
    }\COMMENT{For property \ref{claim1bisitem3}, fix any disjoint sets $A$, $B$ with the correct sizes.
    Then,
    \[\mathbb{P}[e(A,B)\neq(1\pm\delta)|A||B|p]\leq2\exp(-\delta^2|A||B|p/3)=\exp(-\omega(n p^3\log n))=\exp(-\omega(n))=o(3^n).\]
    So the result follows by a union bound on the at most $3^n$ choices for $(A,B)$.
    }\COMMENT{For property \ref{claim1bisitem4}, fix any disjoint sets $A$, $B$ with the correct sizes.
    Then,
    \[\mathbb{P}[e(A,B)\neq(1\pm\delta)|A||B|p]\leq2\exp(-\delta^2|A||B|p/3)=\exp(-\Omega(n^{3/2} p))=\exp(-\omega(n))=o(3^n).\]
    So the result follows by a union bound on the at most $3^n$ choices for $(A,B)$.}

\begin{claim}\label{claim1bis}
A.a.s.\ $G\sim G(n,p)$ satisfies the following properties:
\begin{enumerate}[label=$(\mathrm{\alph*})$]
    \item\label{claim1bisitem1} Every vertex\/ $v\in V(G)$ satisfies that\/ $d(v)=(1\pm\delta)np$. 
    \item\label{claim1bisitem2} For every set $A\subseteq V(G)$ with $|A|\geq \delta n$ we have that $e_G(A)=(1\pm\delta) \inbinom{|A|}{2}p$.
    \item\label{claim1bisitem3} For every pair of disjoint sets\/ $A,B\subseteq V(G)$ with\/ $|A|\geq C\log n$ and $|B|\geq np^2$ we have that\/ $e(A,B)=(1\pm\delta)|A||B|p$.
    \item\label{claim1bisitem4} For every pair of disjoint sets\/ $A,B\subseteq V(G)$ with\/ $|A|,|B|\geq n^{3/4}$ we have that\/ $e(A,B)=(1\pm\delta)|A||B|p$.
\end{enumerate}
\end{claim}

Now, let $G$ be any $n$-vertex graph satisfying the properties of \cref{claim1bis}, and let $\vec{G}$ be an arbitrary orientation of $G$. 
We are going to show that, if $n$ is sufficiently large, then $\vec{G}$ contains a Sullivan vertex.
Let $x_1,\ldots,x_n$ be a labelling of $V(G)$ such that for all $i,j\in[n]$ with $i\leq j$ we have that $d^{+}(x_i)\geq d^{+}(x_j)$; that is, the labels are assigned by decreasing order of the outdegrees of the vertices.
Let $X_0\coloneqq\varnothing$ and, for each $i\in[n]$, let $X_i\coloneqq\{x_j:j\in[i]\}$.

\begin{claim}\label{claim2bis}
If there is some\/ $i\in[n]$ such that\/ $d^+(x_i)\geq(1+\eps)np^2$ and\/ $|N^+(x_i)\cap X_{i-1}|\geq C\log n$, then\/ $x_i$ is a Sullivan vertex.
\end{claim}

\begin{claimproof}
Assume there is some $i\in[n]$ satisfying the conditions.
Let $A\subseteq N^+(x_i)\cap X_{i-1}$ be such that $|A|=C\log n$, and let $B\coloneqq N^+(A)$.
By using the ordering of the vertices and \cref{claim1bis}~\ref{claim1bisitem3},\COMMENT{Notice that we may use \cref{claim1bis}~\ref{claim1bisitem3} because, by simply looking at the ``forward'' outneighbours of the last vertex in $A$, we are guaranteed that $|B|\geq(1+\eps)np^2-C\log n\geq np^2$.} it follows that
\[(1-o(1))(1+\eps)np^2|A|\leq|A|(d^+(x_i)-|A|)\leq\vec{e}(A,B)\leq(1+\delta)|A||B|p,\]
where the first inequality holds by the lower bound on $p$.
It follows that $|B|\geq(1+\delta)np$.\COMMENT{Reordering the expression above, we have that $|B|\geq(1-o(1))(1+\eps)np/(1+\delta)$, so it suffices to verify that $(1-o(1))(1+\eps)\geq(1+\delta)^2$, which is equivalent to $\eps-o(1)\geq2\delta+\delta^2$.
    As $\delta\leq\eps /3$, the conclusion follows.}
As $B\setminus N^+(x_i)\subseteq N_2^+(x_i)$ and $|N^-(x_i)|\leq(1+\delta)np-|N^+(x_i)|$ by \cref{claim1bis}~\ref{claim1bisitem1}, it follows that $x_i$ is a Sullivan vertex.
\end{claimproof}

Next we show that there exist many vertices whose outdegree must be sufficiently large for \cref{claim2bis} to apply to them.

\begin{claim}\label{claim3bis}
    For every $\alpha\in[\delta,1-\delta]$ and every $\beta\in[0,\frac{1+\delta}{1+3\delta}\frac{\alpha}{2}-\frac{\delta}{(1+3\delta)\alpha}]$, there is no set $A\subseteq V(G)$ of size $|A|\geq\alpha n$ such that $d^+(v)<(1+3\delta)\beta np$ for all $v\in A$.
\end{claim}

\begin{claimproof}
    Let us assume for a contradiction that there exists a set $A\subseteq V(G)$ of size $\alpha n$ such that $d^+(v)<(1+3\delta)\beta np$ for all $v\in A$.
    By using the properties of \cref{claim1bis}, we note that
    \begin{align*}
        (1-\delta)n^2p/2\leq\sum_{v\in V(G)}d^+(v)&=\sum_{v\in A}d^+(v)+\sum_{v\in V(G)\setminus A}d^+(v)\\
        &<(1+3\delta)\alpha\beta n^2p+(1+\delta)\binom{(1-\alpha)n}{2}p+(1+\delta)\alpha(1-\alpha)n^2p.
    \end{align*}
    By reordering, one can readily verify that this implies that%
        \COMMENT{We have \[(1+\delta)\binom{(1-\alpha)n}{2}p+(1+\delta)\alpha(1-\alpha)n^2p\leq (1+\delta)(1-\alpha)n^2p\frac{1+\alpha}{2}=\frac{(1+\delta)(1-\alpha^2)}{2}n^2p\]
        so
        \begin{align*}
            \beta &> \frac{1-\delta}{2(1+3\delta)\alpha}-\frac{(1+\delta)(1-\alpha^2)}{2(1+3\delta)\alpha}
            =\frac{1-\delta-(1+\delta-\alpha^2(1+\delta))}{2(1+3\delta)\alpha}=-\frac{\delta}{(1+3\delta)\alpha}+\frac{\alpha(1+\delta)}{2(1+3\delta)}.
        \end{align*}} 
    \[\beta>\frac{1+\delta}{1+3\delta}\frac{\alpha}{2}-\frac{\delta}{(1+3\delta)\alpha},\]
    a contradiction.
\end{claimproof}

Set $\alpha\coloneqq1-2\delta$ and $\beta\coloneqq(1+\eps)p/(1+3\delta)$.
A standard algebraic manipulation shows that $\beta\leq\frac{1+\delta}{1+3\delta}\frac{\alpha}{2}-\frac{\delta}{(1+3\delta)\alpha}$,\COMMENT{Indeed,
\[\beta=\frac{1+\eps}{1+3\delta}p\leq\frac{1+\eps}{1+3\delta}(1/2-\eps)\leq\frac{1+\delta}{1+3\delta}\frac{\alpha}{2}-\frac{\delta}{(1+3\delta)\alpha}\]
(where the first inequality simply uses the assumed upper bound on $p$) if and only if
\begin{align*}
    &\frac{(1+\eps)(1/2-\eps)}{1+3\delta}\leq\frac{(1+\delta)(1-2\delta)}{2(1+3\delta)}-\frac{\delta}{(1+3\delta)(1-2\delta)}\\
    \iff&2(1+\eps)(1/2-\eps)\leq(1+\delta)(1-2\delta)-\frac{2\delta}{1-2\delta}.
\end{align*}
As $1-2\delta\geq1/2$, it suffices to verify that 
\[2(1+\eps)(1/2-\eps)\leq(1+\delta)(1-2\delta)-4\delta\iff1-\eps-2\eps^2\leq1-5\delta-2\delta^2,\]
which holds by the choice of $\delta$.} so we may apply \cref{claim3bis} with these parameters to conclude that
\begin{equation}\label{eq:keypropbisnew}
\text{for all $i\in[2\delta n]$ we have $d^+(x_i)\geq(1+\eps)np^2$.}
\end{equation}
By \cref{claim2bis}, we may thus assume that $\vec{G}$ satisfies that
\begin{equation}\label{eq:keypropbis}
\text{for all $i\in[2\delta n]$ we have $|N^+(x_i)\cap X_{i-1}|\leq C\log n$.}
\end{equation} 

Now let $B\coloneqq X_{2\delta n}$ be partitioned into $B_1\coloneqq X_{\delta n}$ and $B_2\coloneqq B\setminus B_1$.
By \cref{claim1bis}~\ref{claim1bisitem4}, we have that $e(B_1,B_2)\geq(1-\delta)|B_1||B_2|p=(1-\delta)\delta^2 n^2p$.
By \eqref{eq:keypropbis}, for each $v\in B_2$ we have $\vec{e}(v,B_1)\leq C\log n$, and so $\vec{e}(B_1,B_2)\geq(1-2\delta)\delta^2n^2p$.
It follows by averaging that there is some vertex $y\in B_1$ such that $\vec{e}(y,B_2)\geq(1-2\delta)\delta np$.
Let $Y\subseteq N^+(y)\cap B_2$ be a set with $|Y|=(1-2\delta)\delta np$, and let $Z\coloneqq N^+(Y)$.
Observe that $|Z|\geq np^2$ by \eqref{eq:keypropbisnew} and \eqref{eq:keypropbis}, and so, by \cref{claim1bis}~\ref{claim1bisitem4}, $\vec{e}(Y,Z)\leq(1+\delta)|Y||Z|p$.
On the other hand, using \eqref{eq:keypropbisnew} and \cref{claim1bis}~\ref{claim1bisitem2}, we have that $\vec{e}(Y,Z)\geq(1+\eps-\delta+2\delta^2)np^2|Y|$.\COMMENT{We have that the number of edges can be computed as the sum of the outdegrees minus the number of edges that stay inside $Y$.
That is,
\begin{align*}
    \vec{e}(Y,Z)&=\sum_{v\in Y}d^+(v)-e(Y)\geq|Y|(1+\eps)np^2-(1+\delta)|Y|^2p/2\geq|Y|(1+\eps)np^2-|Y|^2p\\
    &=|Y|((1+\eps)np^2-(1-2\delta)\delta np^2)=(1+\eps-\delta+2\delta^2)np^2|Y|,
\end{align*}
where the first inequality uses \eqref{eq:keypropbisnew} and \cref{claim1bis}~\ref{claim1bisitem2}.}
Thus, we conclude that $|Z|\geq(1+\delta)np$.\COMMENT{$|Z|\geq \frac{1+\eps-\delta+2\delta^2}{1+\delta}np$ and this is at least $(1+\delta)np$ iff $\eps-\delta +2\delta^2\geq 2\delta +\delta^2$, that is, iff $\eps\geq3\delta-\delta^2$, which holds since $\delta\leq\eps/3$.}
As $Z\setminus N^+(y)\subseteq N_2^+(y)$ and $|N^-(y)|\leq(1+\delta)np-|N^+(y)|$ by \cref{claim1bis}~\ref{claim1bisitem1}, it follows that $y$ is a Sullivan vertex.
\end{proof}

\end{document}